\documentclass[12pt,reqno]{amsart}

\usepackage{amsmath,amsthm,amssymb,comment,fullpage}
\usepackage{times}
\usepackage[T1]{fontenc}
\usepackage{mathrsfs}
\usepackage{latexsym}
\usepackage[dvips]{graphics}
\usepackage{epsfig}
\usepackage{accents}
\usepackage{float}
\usepackage{amsmath,amsfonts,amsthm,amssymb,amscd}
\input amssym.def
\input amssym.tex
\usepackage{color}
\usepackage{hyperref}
\usepackage{url}
\usepackage{breakurl}
\usepackage{comment}
\usepackage{mathtools}
\newcommand{\bburl}[1]{\textcolor{blue}{\url{#1}}}

\newcommand{\E}{\mathbb{E}}

\renewcommand{\E}{\mathbb{E}}

\numberwithin{equation}{section}

\newtheorem{thm}{Theorem}[section]

\newtheorem{lem}[thm]{Lemma}

\theoremstyle{plain}

\newtheorem{rek}[thm]{Remark}

\newcommand\be{\begin{equation}}
\newcommand\ee{\end{equation}}
\newcommand\bea{\begin{eqnarray}}
\newcommand\eea{\end{eqnarray}}
\newcommand\bi{\begin{itemize}}
	\newcommand\ei{\end{itemize}}
\newcommand\ben{\begin{enumerate}}
	\newcommand\een{\end{enumerate}}
\newcommand\bc{\begin{center}}
	\newcommand\ec{\end{center}}
\newcommand\ba{\begin{array}}
	\newcommand\ea{\end{array}}

\newcommand{\ncr}[2]{{#1 \choose #2}}

\newcommand{\twocase}[5]{#1 \begin{cases} #2 & \text{{\rm #3}}\\ #4 &\text{{\rm #5}} \end{cases}}

\newcommand{\hr}[1]{\href{#1}{\url{#1}}}

\newcommand*{\reff}[1]{\hyperref[#1]{\ref{#1}}}

\title{When Rooks Miss: Probability through Chess}

\author{Steven J. Miller}
\email{\textcolor{blue}{\href{mailto:sjm1@williams.edu, Steven.Miller.MC.96@aya.yale.edu}{sjm1@williams.edu,Steven.Miller.MC.96@aya.yale.edu}}}
\address{Department of Mathematics and Statistics, Williams College, Williamstown, MA 01267}

\author{Haoyu Sheng}
\email{\textcolor{blue}{\href{mailto:hs9@williams.edu}{hs9@williams.edu}}}
\address{Department of Mathematics and Statistics, Williams College, Williamstown, MA 01267}

\author{Daniel Turek}
\email{\textcolor{blue}{\href{mailto:dbt1@williams.edu}{dbt1@williams.edu}}}
\address{Department of Mathematics and Statistics, Williams College, Williamstown, MA 01267}

\thanks{The first named author was partially supported by NSF Grant DMS1561945. We thank the referee for helpful comments on an earlier draft, and Francisco Nascimento for pointing out an algebra error in the computation of a lower order term.}



\date{\today}

\begin{document}

\begin{abstract} A famous (and hard) chess problem asks what is the maximum number of safe squares possible in placing $n$ queens on an $n\times n$ board. We examine related problems from placing $n$ rooks. We prove that as $n\to\infty$, the probability rapidly tends to 1 that the fraction of safe squares from a random placement converges to $1/e^2$. Our interest in the problem is showing how to view the involved algebra to obtain the simple, closed form limiting fraction. In particular, we see the power of many of the key concepts in probability: binary indicator variables, linearity of expectation, variances and covariances, Chebyshev's inequality, and Stirling's formula. \end{abstract}

\maketitle
\tableofcontents

\section{Introduction}

Probability as a subject owes a lot to the theory of games; many of its early results were inspired by attempts to analyze the odds of various configurations occurring. What began in contests involving dice or cards has grown tremendously over the centuries, and almost any subject can be the source of a good problem. In this note we concentrate on an extension of a chess question. There is no dearth of great sources for interesting chess problems (see for example \cite{Pet1, Pet2, Wat}); we study one of them, the $n$ queens problem, with a twist. The only familiarity with chess we need is that queens move (and hence attack) horizontally, vertically and diagonally while rooks move just horizontally and vertically, and instead of the standard $8 \times 8$ board ours is $n \times n$. We assume the reader is familiar with the basics of probability and combinatorics, specifically that $\ncr{n}{k} = n! / k!(n-k)!$ is the number of ways to choose $k$ objects from $n$ when order does not matter.

The standard definition of the problem is the following: \emph{If we place $n$ queens on an $n\times n$ chessboard, what is the maximum number of pawns that can be placed and not attacked by any queen?} The problem is often tweaked to asking not just for the maximum number of pawns, but how many configurations of $n$ queens realize this. We can immediately convert this to probability questions by asking what is the probability a random configuration leaves the maximum number of squares safe, or what is the expected number of safe squares. While it is trivial to write down an approach to solve the first problem (it can be encoded as a simple linear programming problem), the run-time is tremendous and the optimal number of safe squares is known only for small $n$; starting with $n=1$, the known values are \begin{eqnarray} & &	0,\  0,\  0,\  1,\  3,\  5,\  7,\  11,\  18,\  22,\  30,\  36,\  47,\  56,\  72,\  82,\  97,\  111,\  \nonumber\\ & &  132,\ 145,\  170,\  186,\  216,\  240,\  260,\  290,\  324,\  360,\  381,\  420, \dots\end{eqnarray} though for $n \ge 17$ the values stated are merely probably the correct ones (see \cite{LV, OEIS}. We can convert these to the percentage of the board that is safe in annnnnnnnnnnnnnnnnnnnnnn optimal configuration; the percentage starts at zero and is slowly increasing, reaching about 50\% by the time $n=30$; it can trivially be shown to converge to 100\% in the limit.\footnote{Essentially place the $n$ queens in a $\sqrt{n} \times \sqrt{n}$ square in the bottom right corner, which shows that only on the order of $n^{3/2}$ of the $n^2$ squares are attacked. If instead of queens we place rooks, it is not too hard to show that this is optimal; see Problem \#1368 (by Steven J. Miller and Chenyang Sun) of the Spring 2020 issue of the Pi Mu Epsilon Journal.}


We study a variant which not only turns out to be easy to analyze, but also provides an excellent opportunity to see in action many important concepts in probability, including binary indicator variables, expected values, covariances, Chebyshev's inequality, Stirling's formula ($n! = n^n e^{-n} \sqrt{2\pi n} \left(1 + O(1/n)\right)$), partitioning probabilities, and how to handle algebra to obtain clean limiting answers. (Recall big-Oh notation $f(x) = O(g(x))$ means that there is a $C > 0$ such that $|f(x)| \le C g(x)$.) For more on these and related topics, see for example \cite{Mi}.

\ \\

\noindent \emph{\textbf{Rook Problem:} If $n$ rooks are randomly and uniformly placed on an $n\times n$ chessboard, what is the probability that at least one square is safe, what is the number of safe squares, and what is the distribution of the number of safe squares?}

\ \\

It is of course possible that no squares are safe. If we have an $n\times n$ board and want to place $n$ queens, they could attack every square even if there are both rows and columns with no queens. For example, if $n=4$ place the four queens in the corners; no queens are in rows 2 or 3, or columns 2 or 3, but every square is attacked. The situation is very different for rooks, as rooks can only attack horizontally and vertically, while queens also attack diagonally.

For there to be no safe squares, it must be the case that either each row has a rook, or each column has a rook.\footnote{If there are no rooks in the $i$\textsuperscript{th} column and none in the $j$\textsuperscript{th} row, then the square $(i,j)$ is safe. We can bound the probability of a safe square by doubling this probability; to find the exact probability we would then subtract the probability that each row \emph{and} each column has a rook.} Each rook is in a different column with probability $n^n / (n! \ncr{n^2}{n})$. To see this, we argue by labeling the rooks 1 through $n$, so order matters. There are $n$ choices for the row of each rook, giving us $n^n$ configurations that have a rook in each column. There are $n^2$ squares to place the first rook, $n^2 - 1$ for the second, and so on down to $n^2 - (n-1)$ for the $n$\textsuperscript{th} rook; the product of this is $n! \ncr{n^2}{n}$. As $n! \ncr{n^2}{n} \le (n^2)^n$, we have \be \frac{n^n}{n! \ncr{n^2}{n}} \ \le \ \frac{n^n}{n^{2n}} \  = \ \frac{1}{n^{n}}, \ee which rapidly goes to zero with $n$. Thus the configurations where each row or each column has a rook happen with vanishingly small probability. In particular, this means that with probability quickly approaching 1 we have at least one square safe.


Thus the first part of our question is not too interesting; the second, however, is more involved and it is here that good mathematics and statistics enter.

\ \\

\begin{thm} As $n \to \infty$, the expected number of safe squares in the Rook Problem converges to $n^2/e^2$, so the percentage of safe squares converges to approximately 13.53\%.
\end{thm}

\ \\

The idea of the proof is the following. Let $\mathcal{B}$ be a board configuration (i.e., it is the location of $n$ squares on an $n\times n$ board where the rooks are placed), and consider the binary indicator variables \be \twocase{X_{ij}(\mathcal{B}) \ = \ }{1}{if square $(i,j)$ is safe under $\mathcal{B}$}{0}{otherwise.} \ee Each configuration $\mathcal{B}$ occurs with probability $1/\ncr{n^2}{n}$, as we must choose $n$ of the $n^2$ squares for our $n$ rooks.

We compute \be \nu_n \ = \ \sum_{i,j=1}^n \E[X_{ij}(\mathcal{B})], \ \ \ \sigma_n^2 \ = \ \sum_{ij=1}^n \E[X_{ij}(\mathcal{B}) - \nu_n)^2]; \ee note $\nu_n/n^2$ is the average fraction of safe squares. We then use Chebyshev's inequality to show that the number of safe squares for a random $\mathcal{B}$ is concentrated tightly near the mean, answering our second question. Unfortunately the $X_{ij}$ are not independent, and thus we cannot use the variance of a sum is the sum of the variance. Instead, we expand by using covariances.

In \S\ref{sec:combprelim} we isolate and prove a general combinatorial lemma on ratios of binomial coefficients which is used in calculating the mean and the standard deviation. Doing so allows us to do the computation once, and then we just need to change the values of the parameters for the two cases. We go through the argument in complete detail; this is a major motivation of this paper, as one of our goals is to provide commentary on a very difficult part of mathematics: how does one do the algebra to obtain a nice, simple closed form expression? Such simplifications are not always possible, but when they are there is tremendous value, as we now have exact answers and do not need to resort to simulations. After this, in \S\ref{sec:mainproof} we prove our main result about the limiting percentage of safe squares. We also look at our main results using simulation.

\section{Combinatorial Preliminaries}\label{sec:combprelim}

Our key ingredient in the analysis is the following combinatorial lemma involving ratios of binomial coefficients, which we use to calculate the probability of certain rook configurations. We have $n^2$ squares initially available for the rooks, and will see later as we go through all the cases that there are always integers $a$ and $b$ such that the number of squares we lose is $an+b$.

\begin{lem}\label{lem:keybinomialratio} For any positive integer $a$ and any integer $b$,
\be \lim_{n\rightarrow \infty} \frac{\ncr{n^2 - an -b}{n}}{\ncr{n^2}{n}}\ =\ \frac{1}{e^a}.\ee
\end{lem}

\begin{proof} We have
\begin{align}
    \frac{\ncr{n^2 - an -b}{n}}{\ncr{n^2}{n}} &\ = \  \frac{(n^2 - an - b)!}{n!(n^2 - (a+1)n - b)!}  \cdot  \frac{n!(n^2-n)!}{(n^2)!} \nonumber\\
    &\ = \  \frac{(n^2 - an - b)!}{(n^2 - (a+1)n - b)!}  \cdot  \frac{(n^2-n)!}{(n^2)!} \nonumber\\
    &\ = \  \frac{(n^2 - n)(n^2 - n - 1) \cdots (n^2 - (a+1)n - (b-1))}{n^2 (n^2 - 1) \cdots  (n^2 - an - (b-1))} \nonumber\\
    &\ = \ \prod_{k=0}^{an+b-1}\frac{n^2 - n - k}{n^2 - k} \ = \  \prod_{k=0}^{an+b-1}\left(1-\frac{n}{n^2 - k}\right).
\end{align}

The challenge now is to manipulate the product so that we can obtain a nice limit as $n\to\infty$. Notice we have on the order of $n$ terms, each within a fixed multiple of $1/n$ from 1. Thus we should be thinking that the answer is related to the exponential function, as \be e^x \ = \ \lim_{n\to\infty} \left(1 + \frac{x}{n}\right)^n.\ee This \emph{suggests} how we should attack the algebra; we want to view each term $1 - n/(n^2-k)$ as $1 - 1/n$ plus a \emph{very} small correction, a correction significantly smaller than $1/n$ and so small that it will not contribute as $n$ tends to infinity, even when raised to the $n$\textsuperscript{th} power. Once we realize this, the following expansion is forced upon us:
\begin{align}
    \frac{n}{n^2 - k} &\ = \  \frac{1}{n} + \frac{k}{n^3 -n k}.
\end{align}

We can often get a good feeling for an expression by looking at extreme cases. Since $k$ is between 0 and $an+b-1$, we have
\begin{align}\label{eq:firstupperlowerbound}
   \left(1-\frac{1}{n} - \frac{an+b-1}{n^3 - an^2 - (b-1)n}\right)^{an+b} \ \leq\ \prod_{k=0}^{an+b-1}\left(1-\frac{n}{n^2 - k}\right) \ \le \  \left(1-\frac{1}{n}\right)^{an+b}.
\end{align}
Since $a$ and $b$ are fixed integers, the far right is easily analyzed:
\begin{align}
  \lim_{n\rightarrow \infty}\left(1-\frac{1}{n}\right)^{an+b} &\ = \  \left(\lim_{n\rightarrow\infty} \left(1-\frac{1}{n}\right)^{n}\right)^a  \cdot  \left(\lim_{n\rightarrow \infty} \left(1-\frac{1}{n}\right)\right)^b \ = \ e^a.
\end{align}

Similar to above, we have a sense of the size of the factor on the far left: $1 - 1/n$. We pull this factor out, as its behavior to the power $an+b$ is the same as our analysis of the upper bound. If we can just show the remaining factor is extremely close to $1$ then the upper and lower bounds for our product will be close, and we can then take limits as $n\to\infty$. Since
\begin{align}
    1-\frac{1}{n} - \frac{an+b-1}{n^3 - an^2 - (b-1)n} &\ = \  \left(1-\frac{1}{n}\right)\left(1 - \frac{an+b-1}{n^3 - an^2 - (b-1)n} \cdot \frac{n}{n-1}\right)\nonumber\\
    &\ = \  \left(1-\frac{1}{n}\right)\left(1 - \frac{an+b-1}{n^2 - an -
    (b-1)} \cdot \frac{1}{n-1}\right),
\end{align}
substituting this into the left of \eqref{eq:firstupperlowerbound} yields
\begin{align}
   \left(1-\frac{1}{n}\right)^{an+b}\left(1 - \frac{an+b-1}{n^2 - an -
    (b-1)} \cdot \frac{1}{n-1}\right)^{an+b}  \ \le \   \prod_{k=0}^{an+b-1}\left(1-\frac{n}{n^2 - k}\right) \ \le \  \left(1-\frac{1}{n}\right)^{an+b}.
\end{align}

We only need to find the limit for $\left(1 - \frac{an+b-1}{n^2 - an - (b-1)} \cdot \frac{1}{n-1}\right)^{an}$ as $n\rightarrow \infty$ to find the limit for the left. As we only care about its limit as $n\to\infty$, we can be a bit crude in approximations in finding upper and lower bounds, so long as they have the same limit.

A good upper bound is trivial:
\begin{align}
   \lim_{n\rightarrow \infty}\left(1 - \frac{an+b-1}{n^2 - an -
    (b-1)} \cdot \frac{1}{n-1}\right)^{an} \ \le \  \lim_{n\rightarrow \infty} 1^{an}\ =\ 1.
\end{align}

We thus want a lower bound that tends to 1 as $n\to\infty$. For large $n$ the factor is essentially $1 - a/n^2$, and thus when we raise this to the $n$\textsuperscript{th} power it should tend to 1 (from the definition of the exponential function, we would need to raise it to something on the order of $n^2$ to have a limit that is not 1). Thus we can be very crude in our bounding.\footnote{We could be more careful, and take logarithms, analyze and then exponentiate, but there is no need.} Fix any $c > 0$. Since $a$ is a positive integer, as $n\rightarrow \infty$ we have $\frac{an+b-1}{n^2 - an - (b-1)} \leq \frac{c}{n}$. Thus
\begin{align}
    \lim_{n\rightarrow \infty}\left(1 - \frac{an+b-1}{n^2 - an -
    (b-1)} \cdot \frac{1}{n-1}\right)^{an}\ > \ \lim_{n\rightarrow \infty} \left(1- \frac{c}{n}\right)^{an} \ = \ \left(\lim_{n\to\infty}\left(1 - \frac{c}{n}\right)^n\right)^a \ = \ e^{-ca}.
\end{align}
As this is true for any $c>0$, taking the limit as $c$ tends to zero from above we obtain a lower bound of $e^0 = 1$. Since the upper and lower bounds are equal, by the Squeeze Theorem we obtain
\begin{align}
   \lim_{n\rightarrow \infty}\left(1 - \frac{an+b-1}{n^2 - an -
    (b-1)} \cdot \frac{1}{n-1}\right)^{an}\ =\ 1,
\end{align} and therefore
\be \lim_{n\rightarrow \infty} \frac{\ncr{n^2 - an -b}{n}}{\ncr{n^2}{n}} = \lim_{n\rightarrow \infty}\prod_{k=0}^{an+b-1}\left(1-\frac{n}{n^2 - k}\right) = \frac{1}{e^a}, \ee
completing the proof. \end{proof}

\section{Proof of Limiting Percentage}\label{sec:mainproof}

Let $S_n(\mathcal{B})$ be the number of safe spaces on an $n\times n$ board when the $n$ rooks are in configuration $\mathcal{B}$ and thus $S_n(\mathcal{B})/n^2$ is the fraction of squares that are safe; therefore \be S_n(\mathcal{B}) \ = \ \sum_{i,j=1}^n X_{ij}(\mathcal{B}). \ee As the $X_{ij}$ are Bernoulli random variables, to calculate their expected values we just need to know the probability it takes on the value 1 (or the value 0). For a square $(i,j)$ to be safe we cannot place a rook in row $i$ or in column $j$; there are $2n-1$ squares in the union of row $i$ and column $j$, and thus the $n$ rooks must be placed in the $n^2 - (2n-1) = (n-1)^2$ remaining spaces on the board. Thus the probability that $(i,j)$ is safe is $\ncr{(n-1)^2}{n} / \ncr{n^2}{n}$, and therefore \bea \E[X_{ij}]  \ = \ 1 \cdot {\rm Prob}(X_{ij} = 1) + 0 \cdot {\rm Prob}(X_{ij} = 0) \ = \ \frac{\ncr{(n-1)^2}{n}}{\ncr{n^2}{n}} \ =: \ \mu_n.\eea

\subsection{Determining the Mean}

The mean of $S_n$ is easily determined (and from that we can get the percentage of safe squares by dividing by $n^2$). By linearity of expectation, and all the random variables being identically distributed,
\be \E[S_n]\ =\ \sum_{i,j=1}^n \E[X_{ij}] \ = \ n^2 \E[X_{11}] \ = \ n^2 \frac{\ncr{(n-1)^2}{n}}{\ncr{n^2}{n}}, \ee and therefore the expected percentage of the $n\times n$ square that is safe is \be \E[S_n / n^2] \ = \  \frac{\ncr{(n-1)^2}{n}}{\ncr{n^2}{n}} \ = \ \mu_n. \ee

We now use Lemma \ref{lem:keybinomialratio}; as $(n-1)^2 = n^2 - 2n + 1$ we have $a=2$ and $b=-1$, and thus \be \mu \ := \ \lim_{n\to\infty} \mu_n \ = \ \lim_{n\to\infty} \E[S_n / n^2] \ = \  \lim_{n\rightarrow \infty} \frac{\ncr{n^2-2n+1}{n}}{\ncr{n^2}{n}}\ =\ \frac{1}{e^2} \ \approx \ 13.53\%. \ee

It is remarkable that all the binomial ratios interact beautifully and a clean, simple closed form solution is available for the limiting percentage of safe squares when we randomly and uniformly place $n$ rooks on an $n\times n$ board: $\mu = 1/e^2$.

We study this limiting percentage through simulation, and indeed find this limiting percentage of approximately 13.53\% (Figure \ref{fig:percentage_free}).

\begin{figure}[h]
\centering
\includegraphics[scale=0.6]{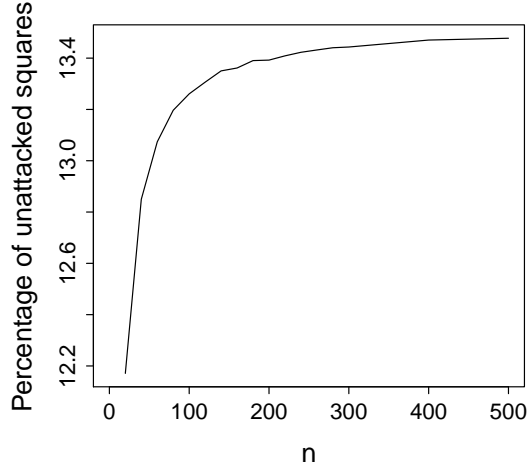}
\caption{Mean percentage of unattacked squares, on an $n \times n$ board.}
\label{fig:percentage_free}
\end{figure}

\subsection{Determining the Variance}

We now turn to the next part: how does the percentage of safe squares vary around $1/e^2$? To determine the answer, we need to compute the variance of $S_n/n^2$ or $S_n$:
\begin{align}\label{eq:varcovarexpansion}
    \text{Var}(S_n) &\ = \  \sum_{i,j = 1}^n \text{Var}(X_{ij}) + \sum_{i,j,\ell,m=1 \atop (i,j) \neq (\ell,m)}^n\text{Cov}(X_{ij}, X_{\ell m}).
\end{align}

We analyze the two pieces separately.

\subsubsection{The Variance Term}

From the definition of variance,
\begin{align}
   \text{Var}\left(X_{ij}\right) &\ = \  \E[X_{ij}^2] - \E[X_{ij}]^2.
\end{align} As the $X_{ij}$ are binary indicator variables, $\E[X_{ij}^2] = \E[X_{ij}]$ and thus
\begin{align}
   \text{Var}\left(X_{ij}\right) &\ = \  \E[X_{ij}] - \E[X_{ij}]^2 \  = \ \mu_n - \mu_n^2.
\end{align}  Thus the contribution of these terms to the variance of $S_n/n^2$ is \be \frac1{n^2} \sum_{i,j = 1}^n \text{Var}(X_{ij}) \ = \ \frac1{n^2} \cdot n^2 (\mu_n - \mu_n^2) \ = \ \mu_n - \mu_n^2,\ee which in the limit as $n\to\infty$ is just \be \mu - \mu^2 \ = \ \frac1{e^2} - \frac1{e^4}. \ee

\subsubsection{The Covariance Term}

Breaking down the covariances, we have
\begin{align}
    \text{Cov}\left(X_{ij}, X_{\ell m}\right) &\ = \  \E[\left(X_{ij} - \mu_{ij}\right)\left(X_{\ell m} - \mu_{\ell m}\right)]\nonumber\\
    &\ = \  \E[X_{ij}X_{\ell m}] - \E[\mu_{\ell m}X_{ij}] - \E[\mu_{ij}X_{\ell m}] + \mu_{ij}\mu_{\ell m}.
\end{align}

As the random variables are identically distributed and all have expected value $\mu_n$,
\begin{align}
    \text{Cov}\left(X_{ij}, X_{\ell m}\right) &\ = \  \E[X_{ij}X_{\ell m}] - \mu_n^2.
\end{align}

As our variables are binary, note $X_{ij}X_{\ell m}(b)$ can be understood as
\be
\twocase{X_{ij}X_{\ell m}(\mathcal{B})\ = \ }{1}{if position $(i, j)$ and position $(\ell,m)$ are safe in board configuration $\mathcal{B}$}{0}{otherwise.} \ee

There are thus two cases to consider for the calculation of $\E[X_{ij}X_{lm}]$.

\ \\

\begin{lem}\label{lem:pairsoverlap} Assume the pairs $(i,j)$ and $(\ell, m)$ share either the same row or the same column (they cannot share both as they are distinct points, since $(i,j) \neq (\ell, m)$). Then \be \text{Prob}\left(X_{ij}X_{lm} = 1\right)\ =\ \E[X_{ij} X_{\ell m}] \ = \ \frac{\ncr{(n-1)(n-2)}{n}}{\ncr{n^2}{n}}, \ee and the number of such pairs is $2(n-1)n^2$. \end{lem}

\begin{proof}
If the two pairs share the same row, there are $\ncr{n}{1}$ ways to choose the row and $\ncr{n}{2}$ ways to choose the column. There are $2!$ ways to assign the columns to $(i,j)$ and $(\ell, m)$, and thus there are $2 \cdot n \cdot n(n-1)/2 = n^2(n-1)$ such pairs. There are the same number of pairs that share a column but not a row, and thus the number of pairs in this case is $2(n-1)n^2$.

What is the probability that $X_{ij} X_{\ell m}(\mathcal{B})$ equals 1 for such a pair? For convenience, let's assume they share a row (so $i=\ell$). We need to place $n$ rooks, none can be in row $i$, and none can be in column $j$ or $m$. We thus have $(n-1) \cdot (n-2)$ squares on the board where we may place the rooks, and the probability all $n$ rooks are in these squares is $\ncr{(n-1)(n-2)}{n} / \ncr{n^2}{n}$. Thus \be \E[X_{ij} X_{\ell m}] \ = \ \frac{\ncr{(n-1)(n-2)}{n}}{\ncr{n^2}{n}}, \ee as claimed, and the contribution from all these pairs is \be 2(n-1)n^2 \frac{\ncr{(n-1)(n-2)}{n}}{\ncr{n^2}{n}} \ = \ 2(n-1)n^2 \frac{\ncr{n^2-3n+2}{n}}{\ncr{n^2}{n}}. \ee
\end{proof}

\begin{lem}\label{lem:pairsnooverlap} Assume the pairs $(i,j)$ and $(\ell, m)$ are in distinct rows and columns (so $i \neq \ell$ and $j \neq m$). Then \be
    \text{Prob}\left(X_{ij}X_{lm} = 1\right)\ =\ \E[X_{ij}X_{\ell m}]\ =\ \frac{\ncr{(n-2)^2}{n}}{\ncr{n^2}{n}},
 \ee and the number of such pairs is $n^2 (n-1)^2$. \end{lem}

\begin{proof} There are $n^2$ ways to choose $(i,j)$, and then $(n - 1)^2$ choices for $(\ell, m)$ so that it is not in the same row or column as $(i,j)$. What is the probability that $X_{ij} X_{\ell m}(\mathcal{B})$ equals 1 for such a pair? We need to place $n$ rooks, none can be in row $i$ or $\ell$, and none can be in column $j$ or $m$. We thus have $(n-2) \cdot (n-2)$ squares on the board where we may place the rooks, and the probability all $n$ rooks are in these squares is $\ncr{(n-2)(n-2)}{n} / \ncr{n^2}{n}$. Thus \be \E[X_{ij} X_{\ell m}] \ = \ \frac{\ncr{(n-2)(n-2)}{n}}{\ncr{n^2}{n}}, \ee as claimed, and the contribution from all these pairs is \be n^2(n-1)^2 \frac{\ncr{(n-2)(n-2)}{n}}{\ncr{n^2}{n}} \ = \ n^2(n-1)^2 \frac{\ncr{n^2-4n+3}{n}}{\ncr{n^2}{n}}. \ee
\end{proof}

\begin{rek}
As a consistency check, note the number of pairs from Lemmas \ref{lem:pairsoverlap} and \ref{lem:pairsnooverlap} is \be 2(n-1)n^2 + n^2 (n-1)^2 \ = \ n^2\left((2n-2) + (n^2 - 2n + 1)\right) \ = \ n^2 (n^2 - 1); \ee as expected, this equals the number of distinct pairs $(i,j)$ and $(\ell, m)$.
\end{rek}

We can now determine the contribution of the covariance terms in \eqref{eq:varcovarexpansion}:
\begin{align}
    \sum_{i,j, \ell, m = 1 \atop (i,j) \neq (\ell,m)}^n \text{Cov}\left(X_{ij}, X_{\ell m}\right) &\ = \   \sum_{i,j, \ell, m = 1 \atop (i,j) \neq (\ell,m)}^n \left(\E[X_{ij}\E_{lm}] - \mu_n^2\right) \nonumber\\
    &\ = \  2(n-1)n^2 \frac{\ncr{(n-1)(n-2)}{n}}{\ncr{n^2}{n}} + n^2(n-1)^2\frac{\ncr{(n-2)^2}{n}}{\ncr{n^2}{n}} - n^2  \cdot  (n^2 - 1)\mu_n^2.
\end{align}

\subsection{Variance of $S_n$}\label{sec:variance}

Using the results from the previous subsections, we can write $\text{Var}(S_n)$ as
\begin{align}
    \text{Var}(S_n) &\ = \  \sum_{i = 1}^n\sum_{j=1}^n\text{Var}(X_{ij}) + \sum_{i=1}^n\sum_{j=1}^n \sum_{l=1}^n\sum_{m=1, i,j \neq l,m}^n\text{Cov}(X_{ij}, X_{lm}) \nonumber\\
    &\ = \  n^2(\mu_n - \mu_n^2) + 2(n-1)n^2 \frac{\ncr{(n-1)(n-2)}{n}}{\ncr{n^2}{n}} + n^2(n-1)^2\frac{\ncr{(n-2)^2}{n}}{\ncr{n^2}{n}} - n^2  \cdot  (n^2 - 1)\mu_n^2  \nonumber\\
    &\ = \  n^2(\mu_n - \mu_n^2) + (2n^3 - 2n^2) \frac{\ncr{(n-1)(n-2)}{n}}{\ncr{n^2}{n}} + (n^4-2n^3 + n^2)\frac{\ncr{(n-2)^2}{n}}{\ncr{n^2}{n}} - (n^4 - n^2)\mu_n^2
\end{align}

As ${\rm Var}(S_n/n^2) = {\rm Var}(S_n)/n^4$, if we divide both sides by $n^4$ we obtain
\begin{eqnarray}
 & & \text{Var}(S_n/n^2)\nonumber\\
&    & = \  \frac{\mu_n - \mu_n^2}{n^2} + \left(\frac{2}{n} - \frac{2}{n^2}\right) \frac{\ncr{n^2 - 3n + 2}{n}}{\ncr{n^2}{n}} + \left(1-\frac{2}{n} + \frac{1}{n^2}\right)\frac{\ncr{n^2-4n+4}{n}}{\ncr{n^2}{n}} - \left(1 - \frac{1}{n^2}\right)\mu_n^2. \nonumber\\
\end{eqnarray}

We do not need to determine the variance perfectly, only bound it well enough so we can apply Chebyshev's inequality. We use Lemma \ref{lem:keybinomialratio} to evaluate the two ratios of Binomial coefficients; in the first we take $a=3$ and $b=-2$, for a limit of the ratio of $1/e^3$; for the second term we have $a=4, b=-4$ and get $1/e^4$. The product $\mu_n^2$ converges to $1/e^2$.

Substituting, we see that the constant term (in $n$) above is zero due to cancelation, and the coefficient of the $1/n$ term, as $n\to\infty$, converges to $2/e^3 - 2/e^4$. As the remaining terms are of size $1/n^2$, we see that for large $n$ \be {\rm Var}(S_n/n^2) \ \approx \ \frac{2}{e^3} \left(1 - \frac1{e}\right) \frac{1}{n}; \ee in particular, for all $n$ sufficiently large we have \be \sigma_n^2 \ := \ {\rm Var}(S_n/n^2) \ \le \ \frac{2(e-1)}{e^3} \frac1{n} \ := \ \frac{C}{n}. \ee

By Chebyshev's inequality, \be {\rm Prob}\left(\left|\frac{S_n}{n^2} - \mu_n\right| \ge k \sigma_n\right) \ \le \ \frac1{k^2}. \ee If we take $k = \log n$ then the probability above converges to zero, and $k \sigma_n \le C \log n / n$ also converges to zero. This means for any fixed $\epsilon' > 0$ that as $n\to\infty$ the probability of $S_n/n^2$ being more than $\epsilon'$ from $\mu_n$ tends to zero; as $\mu_n \to \mu = 1/e^2$, we see that in the limit for almost all boards the percentage of safe squares can be made to be as close to $1/e^2$ as we wish.

\section{Future Work}

The result from \S\ref{sec:variance} is an example of the power of Chebyshev's inequality. Simply by bounding the variance we are able to show that $S_n/n^2$ concentrates on the mean; unfortunately it is not strong enough to say anything about the fluctuations about the mean.

We again turn to simulation, to examine the nature of these fluctuations about the mean.  Simulations suggest that the distribution of free squares is asymptotically Gaussian (Figure \ref{fig:fluctuations}), however we do not attempt to prove this result.

\begin{figure}[h]
\centering
\includegraphics[scale=0.8]{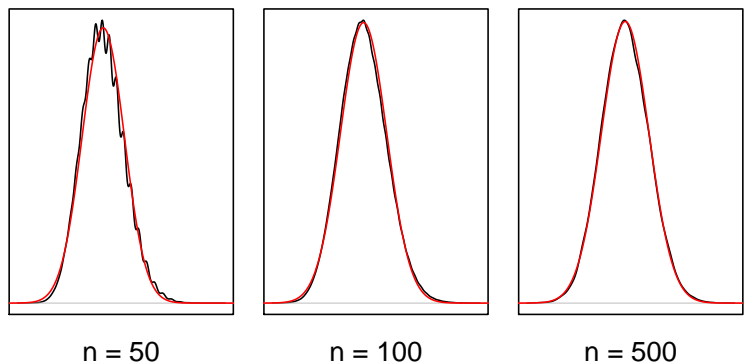}
\caption{Distribution of the percentage of free squares for various board sizes $n$.  Empirical distribution from 10,000 simulations (black), Gaussian density of same mean and variance (red).}
\label{fig:fluctuations}
\end{figure}

There are many other problems one could investigate. For example, what if instead of queens and rooks one were to use bishops? Is it possible to obtain similar results, and if so what combinatorial estimates are now needed?


\ \\

\end{document}